\documentclass[12pt, A4paper]{article}
\usepackage{amsmath, amssymb, amsthm, psfrag, array, graphicx, url}
\theoremstyle{plain}
\newtheorem{thm}{\textrm{Theorem}}
\newtheorem*{thm1}{\textrm{Main Theorem}}
\newtheorem{prop}[thm]{\textrm{Proposition}}
\author{Kuo-Wei Lee}
\title{Constant mean curvature foliation properties in the extended Reissner-Nordstr\"{o}m spacetimes}
\date{\today}
\begin{document}
\fontsize{12}{18pt plus.5pt minus.4pt}\selectfont
\maketitle
\begin{abstract}
We solve the spacelike, spherically symmetric, constant mean curvature hypersurfaces in the maximally extended Reissner-Nordstr\"{o}m spacetime
with the charge smaller than the mass.
Based on these results, we construct constant mean curvature foliations with fixed or varied mean curvature in each slice in this spacetime.
\end{abstract}

\section{Introduction}
In general relativity, spacelike constant mean curvature ({\small CMC}) hypersurfaces in spacetimes are important geometric objects.
Not only {\small CMC} hypersurfaces are used in the analysis on Einstein constraint equations \cite{CYY, L}
and in the gauge condition in the Cauchy problem of the Einstein equations \cite{AM, CYR},
but also the {\small CMC} foliation in cosmological spacetime is identified as the absolute time function \cite{Y}.
Some introduction to the {\small CMC} foliation theory can be found in \cite{KWL2} and references therein.

In this article we are interested in studying spacelike, spherically symmetric,
constant mean curvature ({\small SS-CMC})
hypersurfaces in the maximally extended Reissner-Nordstr\"{o}m spacetime with the charge smaller than the mass and finding {\small CMC}
foliations in this spacetime.
The Reissner-Nordstr\"{o}m spacetime is a static solution of the Einstein-Maxwell field equations.
The Reissner-Nordstr\"{o}m metric reduces to the Schwarzschild metric if there is no charge.
Since we have known many results on the {\small SS-CMC} hypersurfaces and {\small CMC} foliation properties in the extended Schwarzschild spacetime
(Kruskal extension) \cite{KWL, KWL2, LL1, LL2, MO, MO2},
they are very helpful to explore {\small CMC} related questions in the Reissner-Nordstr\"{o}m spacetime.
Remark that Tuite and \'{O} Murchadha also studied spacelike {\small CMC} hypersurfaces in the Reissner-Nordstr\"{o}m spacetime \cite{TO}.
Here we provide different points of view and show more properties in this topic.

To sum up, we answer some {\small CMC} hypersurfaces related questions in the Reissner-Nordstr\"{o}m spacetime.
These results are summarized as the following main theorem.

\begin{thm1}
Consider the maximally extended Reissner-Nordstr\"{o}m spacetime with the charge smaller than the mass.
\begin{itemize}
\item[\rm(a)] The initial value problem for the spacelike, spherically symmetric, constant mean curvature hypersurface equation
in the maximally extended Reissner-Nordstr\"{o}m spacetime is solvable and the solution is unique.
\item[\rm(b)] The Dirichlet problem for the spacelike, spherically symmetric, constant mean curvature hypersurface equation
in the maximally extended Reissner-Nordstr\"{o}m spacetime is solvable and the solution is unique.
\item[\rm(c)] Given $H\in\mathbb{R}$, we can construct a {\small CMC} foliation in the maximally extended Reissner-Nordstr\"{o}m spacetime
so that each slice has the constant mean curvature $H$.
\item[\rm(d)] There is a {\small CMC}
foliation with mean curvature in each slice increasing along the future time direction in the maximally extended Reissner-Nordstr\"{o}m spacetime.
\end{itemize}
\end{thm1}
\noindent More precise statements of these results can be seen in Theorem~\ref{IVP}--\ref{CMCfoliationFixH} and Theorem~\ref{CMCvariedHfoliation}.

The idea for proving these results comes from the analyses when dealing with {\small CMC} hypersurfaces in the Schwarzschild spacetime \cite{KWL, KWL2, LL1, LL2}.
Although this main theorem can be viewed as a generalization of theorems in the Schwarzschild spacetime,
we should mention the essential differences and difficulties between the Reissner-Nordstr\"{o}m spacetime and the Schwarzschild spacetime.
One main difference is that the maximally extended Reissner-Nordstr\"{o}m spacetime is formed by infinitely many Reissner-Nordstr\"{o}m spacetimes.
Compared to the extended Schwarzschild spacetime (Kruskal extension), it is formed by gluing only two Schwarzschild spacetimes.
Because we focus on the Reissner-Nordstr\"{o}m spacetime with the charge smaller than the mass,
every Reissner-Nordstr\"{o}m spacetime in standard coordinates is divided into three regions, $0<r<r_{-}, r_{-}<r<r_{+}$, and $r>r_{+}$.
The mechanism to glue infinitely many Reissner-Nordstr\"{o}m spacetimes is through the region $0<r<r_{-}$,
but no such region is in the Schwarzschild spacetime.

In order to construct a {\small CMC} foliation with fixed mean curvature in each slice,
once we get a {\small CMC} foliation in a pair of Reissner-Nordstr\"{o}m spacetimes,
we can copy the same {\small CMC} foliation to other pairs and get the whole {\small CMC} foliation in the maximally extended Reissner-Nordstr\"{o}m spacetimes.
However, if we want to find a {\small CMC} foliation with varied mean curvature in each slice,
we have to select different slices on different pairs of spacetimes so that the mean curvature is increasing along the future time direction.
This will be the essential difficulties in this topic.
Since different Reissner-Nordstr\"{o}m spacetime regions cause some different situations,
we separately estimate each term of the derivative of a function to show the positivity or negativity,
which is equivalent to the increasing mean curvature property in different regions.
Fortunately, we can construct a {\small CMC} foliation with required properties in the maximally extended Reissner-Nordstr\"{o}m spacetime.

The organization of this paper is as follows.
Section~\ref{Preliminary} is a brief introduction to the Reissner-Nordstr\"{o}m spacetime and its Penrose diagram.
We describe all spacelike, spherically symmetric,
constant mean curvature hypersurfaces in the extended Reissner-Nordstr\"{o}m spacetime and solve the initial value problem for the constant mean curvature equation
in section~\ref{IVPSSCMC}.
Dirichlet problem for the constant mean curvature equation is discussed in section~\ref{DirichletProblem}.
In section~\ref{CMCfoliationConstruction}, we will construct two types of constant mean curvature foliations in the extended Reissner-Nordstr\"{o}m spacetime.
Finally, we summarize results of this paper in section~\ref{Conclusion} and put some detailed computations in Appendix.

\section{The Reissner-Nordstr\"{o}m spacetime} \label{Preliminary}
In this section, we will briefly introduce the Reissner-Nordstr\"{o}m spacetime and its Penrose diagram.
The Reissner-Nordstr\"{o}m spacetime is a $4$-dimensional time-oriented Lorentzian manifold equipped with the metric
\begin{align*}
\mathrm{d}s^2=-\left(1-\frac{2M}r+\frac{e^2}{r^2}\right)\mathrm{d}t^2+\frac1{\left(1-\frac{2M}r+\frac{e^2}{r^2}\right)}\,\mathrm{d}r^2
+r^2\,\mathrm{d}\theta^2+r^2\sin^2\theta\,\mathrm{d}\phi^2,
\end{align*}
where $M>0$ represents the gravitational mass, and $e$ is the electric charge.
When $e=0$, the Reissner-Nordstr\"{o}m spacetime reduces to the Schwarzschild spacetime.
Here we denote $h(r)=1-\frac{2M}r+\frac{e^2}{r^2}$ in convenience.

In this article, we will focus on the case $e^2<M^2$. It implies $h(r)=0$ has two real roots $r_{\pm}=M\pm\sqrt{M^2-e^2}$,
and the metric is regular in region {\tt I} ($r>r_{+}$), region {\tt I\!I} ($r_{-}<r<r_{+}$), and region {\tt I\!I\!I} ($0<r<r_{-}$).
When $r=r_{\pm}$, they are coordinates singularities. The metric can be analytically extended at $r=r_{\pm}$.
The maximally extended Reissner-Nordstr\"{o}m spacetime can be constructed by connecting infinitely many Reissner-Nordstr\"{o}m spacetimes.
Its Penrose diagram is illustrated in Figure~\ref{PDRN},
where regions {\tt I'}, {\tt I\!I'}, and {\tt I\!I\!I'} come from another Reissner-Nordstr\"{o}m spacetime and we upside down of the spacetime to combine them.
We refer to Hawking and Ellies' book \cite[pages 156--159]{HE} for more discussions on the construction of the maximally extended Reissner-Nordstr\"{o}m spacetime.

Here we set $X$-axis and $T$-axis in this Penrose diagram and choose $\partial_T$
as the future directed timelike vector field in the maximally extended Reissner-Nordstr\"{o}m spacetime.

\begin{figure}[h]
\centering
\psfrag{X}{$X$}
\psfrag{T}{$T$}
\psfrag{A}{\tt I}
\psfrag{B}{\tt I\!I}
\psfrag{C}{\tt I\!I\!I}
\psfrag{a}{\tt I'}
\psfrag{b}{\tt I\!I'}
\psfrag{c}{\tt I\!I\!I'}
\psfrag{i}{\footnotesize$i^{-}$}
\psfrag{j}{\footnotesize$i^{0}$}
\psfrag{I}{\footnotesize$i^{+}$}
\psfrag{P}{$r=0$}
\psfrag{Q}{$r=r_{+}$}
\psfrag{R}{$r=r_{-}$}
\psfrag{S}{$r=\infty$}
\includegraphics[height=120mm, width=62mm]{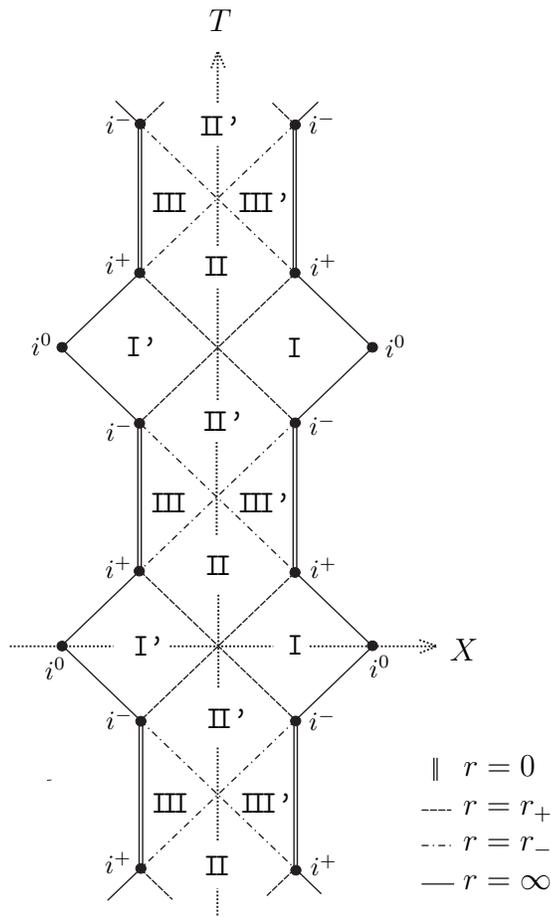}
\caption{Penrose diagram for the maximally extended Reissner-Nordstr\"{o}m spacetime with $e^2<M^2$.} \label{PDRN}
\end{figure}

\section{Initial value problem for the SS-CMC equation in the Reissner-Nordstr\"{o}m spacetime} \label{IVPSSCMC}
Now we are interested in the spacelike, spherically symmetric,
constant mean curvature ({\small SS-CMC}) hypersurfaces in the maximally extended Reissner-Nordstr\"{o}m spacetime.
All {\small SS-CMC} hypersurfaces are solved and these constructions are similar to the constructions in the Schwarzschild spacetime \cite{LL1}.
First, we solve the {\small SS-CMC} equation in each Reissner-Nordstr\"{o}m spacetime region,
and then we determine the correct parameters to smoothly glue two {\small SS-CMC} hypersurfaces at every coordinate singularity
so that the joined hypersurface is the {\small SS-CMC} hypersurface defined in the extended Reissner-Nordstr\"{o}m spacetime.
Compared to the Schwarzschild spacetime case, we still need to consider {\small SS-CMC} hypersurfaces in region {\tt I\!I\!I}
and look at the behavior of {\small SS-CMC} hypersurfaces at $r=r_{-}$.
This case is in fact similar to the case in region {\tt I}, and it does not cause any problem.

Here we only state the key ingredients for the construction of {\small SS-CMC} solutions,
and these materials will be used to describe {\small CMC}
foliations in the maximally extended Reissner-Nordstr\"{o}m spacetime in section~\ref{CMCfoliationConstruction}.
Suppose that $\Sigma:(t=f(r),r,\theta,\phi)$ is an {\small SS-CMC} hypersurface in the Reissner-Nordstr\"{o}m spacetime.
The \noindent{\small SS-CMC} equation in the Reissner-Nordstr\"{o}m spacetime is
\begin{align*}
f''+\left(\left(\frac1{h}-(f')^2h\right)\left(\frac{2h}{r}+\frac{h'}{2}\right)+\frac{f'}{h}\right)f'\mp3H\left(\frac1h-(f')^2h\right)^{\frac32}=0,
\end{align*}
where $h(r)=1-\frac{2M}{r}+\frac{e^2}{r^2}$ and it implies $h'(r)=\frac{2M}{r^2}-\frac{2e^2}{r^3}$, and $H$ is the mean curvature.
The spacelike condition is equivalent to $\frac1h-(f')^2h>0$.
Notice that the minus sign (or plus sign) of the term $\mp3H\left(\frac1h-(f')^2h\right)^{\frac32}$
is applied in regions {\tt I}, {\tt I\!I}, {\tt I\!I\!I} (or {\tt I'}, {\tt I\!I'}, {\tt I\!I\!I'}).
The reason for different signs of this term in {\small SS-CMC} equations is that we choose $\partial_T$
as the future timelike direction in the Penrose diagram, and regions {\tt I', I\!I', I\!I\!I'}
come from the upside down of some Reissner-Nordstr\"{o}m spacetime.
It implies that $\partial_T$ points to opposite direction between regions {\tt I} and {\tt I'}, {\tt I\!I} and {\tt I\!I'},
{\tt I\!I\!I} and {\tt I\!I\!I'}, respectively.

Taking region {\tt I\!I} for example,
suppose that {\small SS-CMC} hypersurfaces are piecewisely divided into graphs of functions with positive slope or negative slope
(allowing $f'(r)=+\infty$ or $-\infty$ at some point).
Denote $l(r)=\frac1{\sqrt{-h(r)}}\left(-Hr+\frac{c}{r^2}\right)$,
\footnote{Compared to the paper \cite[page 13, line $-8$, or page 15, line $-8$]{LL1},
here we change the sign of the parameter, that is, let $c=-c_2$.
The reason for changing the sign of this parameter is that it will be easier to describe the {\small CMC}
foliation properties in section~\ref{CMCfoliationConstruction}.}
then
\begin{align*}
f'(r)=\left\{
\begin{array}{ll}
\frac1{-h(r)}\sqrt{\frac{l^2(r)}{l^2(r)-1}} & \mbox{if } f'(r)>0\\[2mm]
\frac1{h(r)}\sqrt{\frac{l^2(r)}{l^2(r)-1}} & \mbox{if } f'(r)<0,
\end{array}\right.
\mbox{and we require } l(r)>1.
\end{align*}
Similarly, in region {\tt I\!I'}, denote $l(r)=\frac1{\sqrt{-h(r)}}\left(Hr-\frac{c}{r^2}\right)$, then
\begin{align*}
f'(r)=\left\{
\begin{array}{ll}
\frac1{-h(r)}\sqrt{\frac{l^2(r)}{l^2(r)-1}} & \mbox{if } f'(r)>0\\[2mm]
\frac1{h(r)}\sqrt{\frac{l^2(r)}{l^2(r)-1}} & \mbox{if } f'(r)<0,
\end{array}\right.
\mbox{and we require } l(r)>1.
\end{align*}
Since the {\small SS-CMC} equation is a second order ordinary differential equation,
the solution of $f'(r)$ gives a freedom $c$, and the solution of $f(r)$ gives another freedom, called $\bar{c}$ for instance.

From this explicit formula, we know that the domain of the function $f(r)$ depends on the parameter $c$.
In region {\tt I\!I}, we get the condition $c>Hr^3+r(-r^2+2Mr-e^2)^{\frac12}\stackrel{\tiny{\mbox{denote}}}{=\!=\!=}F(H,r)$,
so the domain of $f(r)$ is
\begin{align*}
\{r\in(r_{-},r_{+})|c>F(H,r)\}\cup\{r\in(r_{-},r_{+})|c=F(H,r) \mbox{ and } f(r) \mbox{ is finite}\}.
\end{align*}
In region {\tt I\!I'}, we get the condition $c<Hr^3-r(-r^2+2Mr-e^2)^{\frac12}\stackrel{\tiny{\mbox{denote}}}{=\!=\!=}G(H,r)$,
and the domain of $f(r)$ is
\begin{align*}
\{r\in(r_{-},r_{+})|c<G(H,r)\}\cup\{r\in(r_{-},r_{+})|c=G(H,r) \mbox{ and } f(r) \mbox{ is finite}\}.
\end{align*}
Let $C_H=\max\limits_{r\in[r_-,r_+]}F(H,r)=F(H,R_H)$ and $c_H=\min\limits_{r\in[r_-,r_+]}G(H,r)=G(H,r_H)$.
Now we are ready to describe {\small SS-CMC} hypersurfaces in the extended Reissner-Nordstr\"{o}m spacetime.

\begin{itemize}
\item[(A)]
If $c\in(c_H,C_H)$, two graphs of $f(r)$ with positive slope and negative slope can be smoothly glued at the point which satisfies $f'(r)=+\infty$ or $-\infty$
by choosing correct $\bar{c}$,
so the union of two graphs of $f(r)$ forms an {\small SS-CMC} hypersurface in region {\tt I\!I} or {\tt I\!I'}.
This joined {\small SS-CMC} hypersurface will touch the coordinates singularities $r=r_{-}$ or $r=r_{+}$ in the extended Reissner-Nordstr\"{o}m spacetime sense,
and we can find an {\small SS-CMC} hypersurface in different region with the same $c$ and correct $\bar{c}$
to glue them smoothly and finally we get the whole {\small SS-CMC} hypersurface.
In this case, the {\small SS-CMC} hypersurface will range {\tt I'}, {\tt I\!I}, {\tt I} (see Figure~\ref{CMCinRNpaperIVP2} (A)),
or {\tt I\!I\!I}, {\tt I\!I}, {\tt I\!I\!I'}, or
{\tt I'}, {\tt I\!I'}, {\tt I}, or {\tt I\!I\!I}, {\tt I\!I'}, {\tt I\!I\!I'}.
\item[(B)] If $c=C_H$ (or $c=c_H$), $f(r)$ is defined on $(r_{-},R_{H})$ or $(R_{H},r_{+})$ (or $(r_{-},r_{H})$ or $(r_{H},r_{+})$).
By calculating the order of $f'(r)$ near $r=R_{H}$ (or $r=r_{H}$), we know that $f'(r)\in O(|r-R_{H}|^{-1})$ (or $f'(r)\in O(|r-r_{H}|^{-1})$)
so that $\lim\limits_{r\to R_{H}}|f(r)|=\infty$ (or $\lim\limits_{r\to r_{H}}|f(r)|=\infty$).
This {\small SS-CMC} hypersurface will touch coordinate singularities $r=r_{-}$ or $r=r_{+}$,
and we can find an {\small SS-CMC} hypersurface in region {\tt I}, {\tt I\!I\!I}, {\tt I'}, or {\tt I\!I\!I'},
with the same $c$ and correct $\bar{c}$ to glue them smoothly.
In this case, the {\small SS-CMC} hypersurface will range {$i^+$}, {\tt I\!I}, {\tt I} (see Figure~\ref{CMCinRNpaperIVP2} (B)),
or {\tt I'}, {\tt I\!I}, {$i^+$}, or {$i^-$}, {\tt I\!I'}, {\tt I}, or {\tt I'}, {\tt I\!I'}, {$i^-$},
or {\tt I\!I\!I}, {\tt I\!I}, {$i^+$}, or {$i^+$}, {\tt I\!I}, {\tt I\!I\!I'},
or {\tt I\!I\!I}, {\tt I\!I'}, {$i^-$}, or {$i^-$}, {\tt I\!I'}, {\tt I\!I\!I'}.
\item[(C)] If $c>C_H$ (or $c<c_H$), $f(r)$ is defined on $(r_{-},r_{+})$ in region {\tt I\!I} (or {\tt I\!I'}),
Such {\small SS-CMC} hypersurface will touch coordinate singularities $r=r_{-}$ or $r=r_{+}$,
and we can find an {\small SS-CMC} hypersurface in region {\tt I}, {\tt I\!I\!I}, {\tt I'}, or {\tt I\!I\!I'},
with the same $c$ and correct $\bar{c}$ to glue them smoothly.
In this case, the {\small SS-CMC} hypersurface will range {\tt I\!I\!I}, {\tt I\!I}, {\tt I} (see Figure~\ref{CMCinRNpaperIVP2} (C)),
or {\tt I'}, {\tt I\!I}, {\tt I\!I\!I'} , or {\tt I'}, {\tt I\!I'}, {\tt I\!I\!I'}, or {\tt I\!I\!I}, {\tt I\!I'}, {\tt I}.
\item[(D)] Besides {\small SS-CMC} hypersurfaces range over different regions,
there are another types of {\small SS-CMC} hypersurfaces, called cylindrical hypersurfaces.
These hypersurfaces are of the form $(t,r=r_0,\theta,\phi)$, where $r_0\in(r_{-},r_{+})$ in region {\tt I\!I} (see Figure~\ref{CMCinRNpaperIVP2} (D))
or region {\tt I\!I'}
with constant mean curvature
\begin{align*}
H(r)=\frac{\pm1}{3\sqrt{-h(r)}}\left(\frac{2h(r)}{r}+\frac{h'(r)}{2}\right)=\frac{\pm(2r^2-3Mr+e^2)}{3r^2\sqrt{-r^2+2Mr-e^2}}.
\end{align*}
\end{itemize}

\begin{figure}[h]
\centering
\psfrag{X}{$X$}
\psfrag{T}{$T$}
\psfrag{A}{\tt I}
\psfrag{B}{\tt I\!I}
\psfrag{C}{\tt I\!I\!I}
\psfrag{a}{\tt I'}
\psfrag{b}{\tt I\!I'}
\psfrag{c}{\tt I\!I\!I'}
\psfrag{i}{\footnotesize$i^{-}$}
\psfrag{j}{\footnotesize$i^{0}$}
\psfrag{I}{\footnotesize$i^{+}$}
\psfrag{P}{$r=0$}
\psfrag{Q}{$r=r_{+}$}
\psfrag{R}{$r=r_{-}$}
\psfrag{S}{$r=\infty$}
\psfrag{p}{\tiny (A)}
\psfrag{q}{\tiny (B)}
\psfrag{r}{\tiny (C)}
\psfrag{s}{\tiny (D)}
\includegraphics[height=76mm, width=62mm]{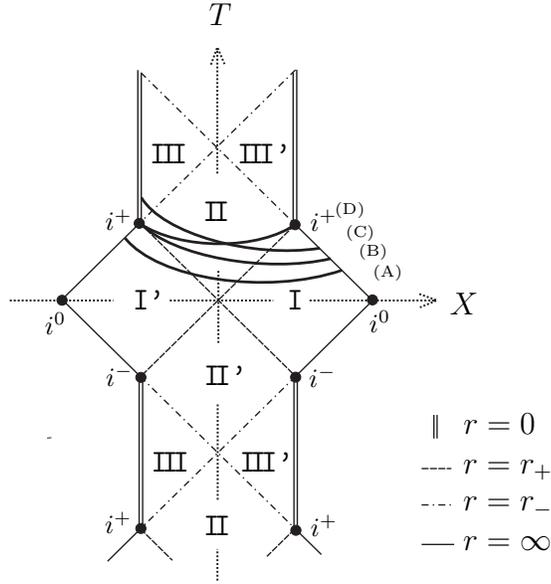}
\caption{Constant mean curvature hypersurfaces with different parameters $c$ are illustrated.
Here we show an example in each case (A), (B), (C), and (D).} \label{CMCinRNpaperIVP2}
\end{figure}

From the above discussion, we can state and prove the theorem about the existence and uniqueness of the initial value problem for the
{\small SS-CMC} equation in the maximally extended Reissner-Nordstr\"{o}m spacetime.
\begin{thm} \label{IVP}
Given $H\in\mathbb{R}$, a point $(T_0,X_0)$, and a value $V$ with $1-V^2>0$ in the Penrose diagram of the Reissner-Nordstr\"{o}m spacetime,
there is a unique function $T=T(X)$ such that $T(X_0)=T_0$, $T'(X_0)=V$, and the hypersurface
$\Sigma:(T=T(X),X,\theta,\phi)$ is spacelike {\rm(}$1-(T'(X))^2)>0${\rm)} with constant mean curvature $H$.
\end{thm}

\begin{proof}
The analyses (A)--(D) indicate that the {\small SS-CMC} equation in the standard coordinates in each region is solvable,
and the initial values $T(X_0)=T_0$ and $T'(X_0)=V$ will uniquely determine $c$ and $\bar{c}$ at the
{\small SS-CMC} hypersurface containing $(T_0,X_0)$.
For {\small SS-CMC} hypersurfaces in different regions,
in order to smoothly glue together, their $c$ and $\bar{c}$ are also uniquely determined.
Finally, since there is a one-to-one correspondence between the standard coordinates in each region and the Penrose diagram,
the theorem is described in the maximally extended Reissner-Nordstr\"{o}m spacetime sense and is proved.
\end{proof}

\section{Dirichlet problem for the SS-CMC equation in the Reissner-Nordstr\"{o}m spacetime} \label{DirichletProblem}
One application of Theorem~\ref{IVP} is to answer the Dirichlet problem for the {\small SS-CMC}
equation in the extended Reissner-Nordstr\"{o}m spacetime.
Here we only focus on the Dirichlet problem with symmetric boundary data because our goal is to construct
$T$-axisymmetric {\small SS-CMC} foliations in next section.

Before describing the Dirichlet problem,
let us go back to the construction (A) and (D) in section~\ref{IVPSSCMC} to collect all $T$-axisymmetric {\small SS-CMC} hypersurfaces.
An {\small SS-CMC} hypersurface $\Sigma$ is determined by two parameters $c$ and $\bar{c}$.
The effect of the parameter $\bar{c}$ in the standard coordinates $(t,r,\theta,\phi)$
means the translation of an {\small SS-CMC} hypersurface $\Sigma$ along the $t$-direction,
and it presents the Lorentzian isometry in the Penrose diagram.
In case (A), given $H\in\mathbb{R}$ and $c\in(c_H,C_H)$, among the family of {\small SS-CMC} hypersurfaces $\{\Sigma_{H,c,\bar{c}}\}_{\bar{c}\in\mathbb{R}}$,
there is a unique hypersurface with $T$-axisymmetric.
We call it {\small TSS-CMC} hypersurface and use the notation $\Sigma_{H,c}$ to represent it.

In the construction (D), if $c=C_H$, the only {\small TSS-CMC} hypersurface is the cylindrical hypersurface $(t,r=R_H,\theta,\phi)$
in region {\tt I\!I}. Similarly, if $c=c_H$, the only {\small TSS-CMC} hypersurface is the cylindrical hypersurface
$(t,r=r_H,\theta,\phi)$ in region {\tt I\!I'}.

Thus, we know all {\small TSS-CMC} hypersurfaces in the extended Reissner-Nordstr\"{o}m spacetime.
Conversely, given two symmetric boundary data in the extended Reissner-Nordstr\"{o}m spacetime,
Dirichlet problem will ask that whether there is a unique {\small TSS-CMC} hypersurface satisfying the boundary data.
Precisely, we have the following theorem.
\begin{thm} \label{DirichletTheorem}
Given $H\in\mathbb{R}$ and boundary data $(T_0,X_0,\theta,\phi)$, $(T_0,-X_0,\theta,\phi)$ in the maximally extended Reissner-Nordstr\"{o}m spacetime,
there exists a unique {\small TSS-CMC} hypersurface $\Sigma:(T=T(X),X,\theta,\phi)$ with mean curvature $H$
satisfying the boundary value conditions $T(X_0)=T(-X_0)=T_0$.
\end{thm}

\begin{proof}
The idea to prove Theorem~\ref{DirichletTheorem} is similar to the Schwarzschild spacetime \cite[pages 9--15]{KWL}.
First of all, we use the shooting method to get the existence of the Dirichlet problem.
Given $H\in\mathbb{R}$, $V$ with $1-V^2>0$, and $(T_0,X_0,\theta,\phi)$,
consider the family of hypersurfaces $\{\Sigma_{H,c(V),\bar{c}(T_0,X_0)}\}_{1-V^2>0}$.
This family collects {\small SS-CMC} hypersurfaces with mean curvature $H$, $T(X_0)=T_0$, and $T'(X_0)=V$.
In this notation $\{\Sigma_{H,c(V),\bar{c}(T_0,X_0)}\}_{1-V^2>0}$,
we just remark that the parameter $c$ is determined by $V$, $\bar{c}$ is determined by $T_0$ and $X_0$,
and $\{\Sigma_{H,c(V),\bar{c}(T_0,X_0)}\}_{1-V^2>0}$ is continuously varied when $V$ changes.

When $V$ tends to $1$, $\Sigma_{H,c(V),\bar{c}(T_0,X_0)}$ tends to some future null cone.
When $V$ goes to $-1$, $\Sigma_{H,c(V),\bar{c}(T_0,X_0)}$ approaches to some past null cone.
Since the other boundary data $(T_0,-X_0,\theta,\phi)$ lies between these future null cone and past null cone,
by the Intermediate Value Theorem,
there is an {\small SS-CMC} hypersurface $\Sigma$ in the family$\{\Sigma_{H,c(V),\bar{c}(T_0,X_0)}\}_{1-V^2>0}$ satisfying $T(-X_0)=T_0$.

This {\small SS-CMC} hypersurface $\Sigma$ is in fact symmetric with respect to the $T$-axis.
We can argue it by contradiction.
Suppose that $\Sigma$ is not $T$-axisymmetric.
Consider its $T$-axisymmetric reflection hypersurface called $\tilde{\Sigma}$.
Two hypersurfaces $\Sigma$ and $\tilde{\Sigma}$ have different parameters $c$
because they have different slopes at the boundary point.
On the other hand, the $T$-axisymmetric reflection does not change the radius of the ``throat'',
that is, the extreme value of the $r$-component function restricted on an {\small SS-CMC} hypersurface.
From the construction of {\small SS-CMC} hypersurfaces (A) and (D) in section~\ref{IVPSSCMC},
we know that the radius of the throat is determined by $c$, so $\Sigma$ and $\tilde{\Sigma}$ must share the same parameter $c$, and it leads to the contradiction.

Now we will give a proof of the uniqueness part.
First, consider the family of hypersurfaces $\{\Sigma_{H,c,\bar{c}}|T(X_0)=T(-X_0)=T_0\}_{H\in\mathbb{R}}$.
This family collects all {\small TSS-CMC} hypersurfaces passing through $(T_0,X_0,\theta,\phi)$ and $(T_0,-X_0,\theta,\phi)$.
Hypersurfaces in this family are continuously changed with respect to $H$.
Next, we argue the uniqueness of the Dirichlet problem by contradiction.
Suppose not, that is, given $H\in\mathbb{R}$, there are two {\small TSS-CMC} hypersurfaces $\Sigma^1:(T=T^1(X),X,\theta,\phi)$ and
$\Sigma^2:(T=T^2(X),X,\theta,\phi)$ satisfying $T^1(X_0)=T^2(X_0)=T_0$, $T^1(-X_0)=T^2(-X_0)=T_0$, and $T^1(X)\not\equiv T^2(X)$ in $(-X_0,X_0)$.
Without loss of generality, assume that two hypersurfaces have causality relation $\Sigma^1\ll\Sigma^2$ in $(-X_0,X_0)$.
We slightly perturb $\Sigma^2$ in the family $\{\Sigma_{H,c,\bar{c}}|T(X_0)=T(-X_0)=T_0\}_{H\in\mathbb{R}}$ to $\Sigma$ so that
$\Sigma^1\ll\Sigma$ in $(-X_0,X_0)$, and $\Sigma$ has constant mean curvature $H+\varepsilon$ for some $\varepsilon>0$.

Consider $\mathrm{d}_{\Sigma^1}|_{\Sigma}:\Sigma\to[0,\infty]$ the Lorentzian distance function with respect to $\Sigma^1$ restricted on $\Sigma$.
We refer \cite{AHP} or \cite[page 7]{KWL} for more discussions on this distance function.
One property is that the Laplacian of $\mathrm{d}_{\Sigma^1}|_{\Sigma}$, called $\Delta_{\Sigma}(\mathrm{d}_{\Sigma^1}|_{\Sigma})$, at each point $q\in\Sigma$
is related to the Laplacian of $\mathrm{d}_{\Sigma^1}$ with respect to the ambient spacetime, the mean curvature of $\Sigma$,
and the Hessian of $\mathrm{d}_{\Sigma^1}$ applying on the unit normal part $\nu$.
More precisely, we have the following equality:
\begin{align*}
\Delta_{\Sigma}(\mathrm{d}_{\Sigma^1}|_{\Sigma})(q)
=\overline{\Delta}\mathrm{d}_{\Sigma^1}(q)+\overline{\mathrm{Hess}}(\mathrm{d}_{\Sigma^1})(q;\nu,\nu)
+3H_{\Sigma}(q)\sqrt{1+|\nabla(\mathrm{d}_{\Sigma^1}|_{\Sigma})(q)|^2}.
\end{align*}

Since $\Sigma^1\neq\Sigma$, we know that the maximum value of $\mathrm{d}_{\Sigma^1}|_{\Sigma}$ is positive and will achieve at some point $q\in\Sigma$.
Since the Reissner-Nordstr\"{o}m spacetime satisfies the timelike convergence condition,
which implies $\overline{\Delta}\mathrm{d}_{\Sigma^1}(q)\geq -3H_{\Sigma^1}(p)$ (see \cite{AHP} or \cite[page 7]{KWL}),
where $p$ is the orthogonal projection of $q$ on $\Sigma^1$,
we get at the maximum point $q$,
\begin{align*}
0\geq\Delta_{\Sigma}(\mathrm{d}_{\Sigma^1}|_{\Sigma})(q)
&\geq\overline{\mathrm{Hess}}(\mathrm{d}_{\Sigma^1})(q;\nu,\nu)+3H_{\Sigma}(q)\sqrt{1+|\nabla(\mathrm{d}_{\Sigma^1}|_{\Sigma})(q)|^2}-3H_{\Sigma^1}(p) \\
&=3H_{\Sigma}(q)-3H_{\Sigma^1}(p)=3\varepsilon>0.
\end{align*}
It leads to a contradiction.
Remark that we use the result $\overline{\mathrm{Hess}}(\mathrm{d}_{\Sigma^1})(q;\nu,\nu)=0$
in the above discussion because of the perpendicular property of maximum distance on two submanifolds \cite[Proposition 2]{KWL}.
Therefore, the uniqueness of the Dirichlet problem is proved.
\end{proof}

\section{Constant mean curvature foliations in the Reissner-Nordstr\"{o}m spacetime} \label{CMCfoliationConstruction}
In this section, we will construct two types of $T$-axisymmetric, spacelike, spherically symmetric, constant mean curvature ({\small TSS-CMC})
foliations in the maximally extended Reissner-Nordstr\"{o}m spacetime.
One foliation has the same constant mean curvature for every hypersurface,
and the other foliation has increasing constant mean curvatures when observing hypersurfaces along the $T$-direction.

Recall that a {\small TSS-CMC} foliation in the extended Reissner-Nordstr\"{o}m spacetime is a family of hypersurfaces
$\{\Sigma_{s}\}$ satisfying (i) every $\Sigma_{s}$ is a {\small TSS-CMC} hypersurface, (ii) for any $s_1\neq s_2$, $\Sigma_{s_1}$ and $\Sigma_{s_2}$
are disjoint, and (iii) $\cup_s\Sigma_s$ covers the whole spacetime.

\begin{thm} \label{CMCfoliationFixH}
Given $H\in\mathbb{R}$, there is a {\small TSS-CMC}
foliation in the maximally extended Reissner-Nordstr\"{o}m spacetime so that every hypersurface has the mean curvature $H$.
\end{thm}

\begin{proof}
Given $H\in\mathbb{R}$, consider the closed loop formed by the graphs of $(r,F(H,r))$ and $(r,G(H,r))$ in the $rc$-plane, where
$F(H,r)=Hr^3+r(-r^2+2Mr-e^2)^{\frac12}$ and $G(H,r)=Hr^3-r(-r^2+2Mr-e^2)^{\frac12}$ are defined in section~\ref{IVPSSCMC}.
Figure~\ref{CMCinRNpaperDomain} illustrates the closed loop when $H>0$ in the $rc$-plane.
We parameterize this closed loop by $\gamma(s)$, where $s\in\mathbb{R}$, $\gamma(0)=(r_{-},F(H,r_{-}))=(r_{-},G(H,r_{-}))$,
and when $s$ increases, the loop $\gamma(s)$ goes counterclockwise.
Since $\gamma(s)$ is a closed loop, we can further set the parameter $s$ with period $T>0$ so that $\gamma(s+kT)=\gamma(s)$ for all $k\in\mathbb{Z}$.

\begin{figure}[h]
\centering
\psfrag{X}{$X$}
\psfrag{T}{$T$}
\psfrag{A}{\tt I}
\psfrag{B}{\tt I\!I}
\psfrag{C}{\tt I\!I\!I}
\psfrag{a}{\tt I'}
\psfrag{b}{\tt I\!I'}
\psfrag{c}{\tt I\!I\!I'}
\psfrag{i}{\footnotesize$i^{-}$}
\psfrag{j}{\footnotesize$i^{0}$}
\psfrag{I}{\footnotesize$i^{+}$}
\psfrag{P}{$r=0$}
\psfrag{Q}{$r=r_{+}$}
\psfrag{R}{$r=r_{-}$}
\psfrag{S}{$r=\infty$}
\psfrag{r}{$r$}
\psfrag{y}{$c$}
\psfrag{M}{\footnotesize$\gamma(0)$}
\psfrag{N}{\footnotesize$\gamma(s_1)$}
\psfrag{Y}{\footnotesize$\gamma(s_2)$}
\psfrag{Z}{\footnotesize$\gamma(s_3)$}
\psfrag{W}{\footnotesize$\gamma(T)$}
\psfrag{m}{\footnotesize$\Sigma_0$}
\psfrag{n}{\footnotesize$\Sigma_{s_1}$}
\psfrag{e}{\footnotesize$\Sigma_{s_2}$}
\psfrag{z}{\footnotesize$\Sigma_{s_3}$}
\psfrag{w}{\footnotesize$\Sigma_T$}
\includegraphics[height=120mm, width=152mm]{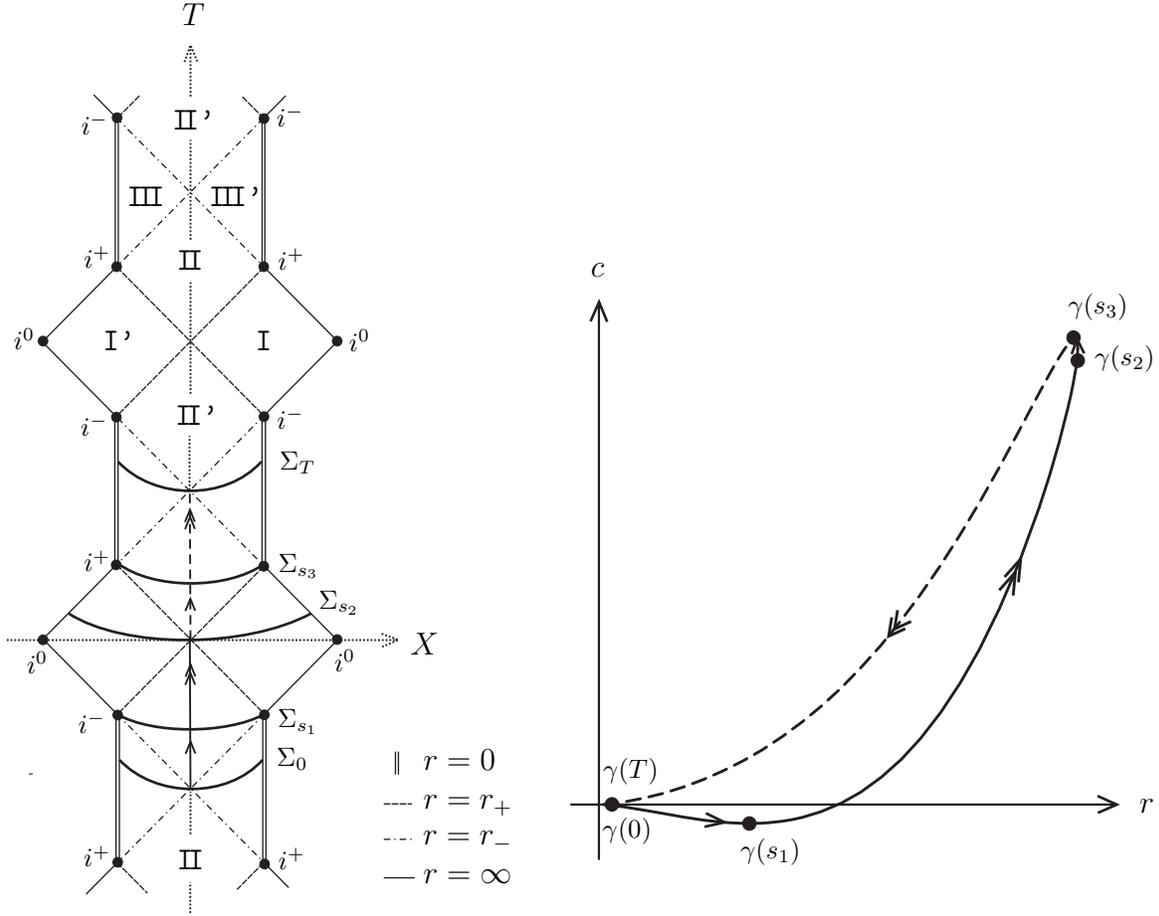}
\caption{Two graphs of functions $(r,F(H,r))$ and $(r,G(H,r))$ form a closed loop in the case $H>0$.
This closed loop is parameterized by $\gamma(s), s\in\mathbb{R}$ with period $T>0$.
Every $\gamma(s)$ corresponds to a {\small TSS-CMC} hypersurface in the extended  Reissner-Nordstr\"{o}m spacetime
so that we construct a {\small TSS-CMC} foliation.} \label{CMCinRNpaperDomain}
\end{figure}

First, we choose a pair of Reissner-Nordstr\"{o}m spacetimes.
One Reissner-Nordstr\"{o}m spacetime consists of regions {\tt I}, {\tt I\!I}, {\tt I\!I\!I},
and another adjacent Reissner-Nordstr\"{o}m spacetime consists of regions {\tt I'}, {\tt I\!I'}, {\tt I\!I\!I'},
and they contain $(T,X,\theta,\phi)=(0,0,\theta,\phi)$.
We will say that for $0\leq s\leq T$,
there is a one-to-one correspondence between every point of $\gamma(s)$
and a {\small TSS-CMC} hypersurface $\Sigma_s$ in this spacelike extension of a pair of Reissner-Nordstr\"{o}m spacetimes.
This is because $r$-component of $\gamma(s)$ tells us the radius of the ``throat" of the {\small TSS-CMC} hypersurface, that is,
the $T$-intercept of the {\small TSS-CMC} hypersurface in the Penrose diagram.
From the theorem of the initial value problem for the {\small SS-CMC} equation,
an {\small SS-CMC} hypersurface is determined by the position of the throat on $T$-axis,
and the uniqueness also implies that this {\small SS-CMC} hypersurface is $T$-axisymmetric.
Hence we establish the one-to-one correspondence relation.

See Figure~\ref{CMCinRNpaperDomain}.
Here we point out some special points on the loop and their correspondences.
The point $\gamma(0)$ corresponds to a {\small TSS-CMC} hypersurface $\Sigma_0$ at the bottom of the pair of Reissner-Nordstr\"{o}m spacetimes.
When $s$ increases, $\gamma(s)$ goes the loop counterclockwise, and the $T$-intercept of the {\small TSS-CMC} hypersurface in Penrose diagram is increasing.
The point $\gamma(s_1)$ achieves the minimum value of the function $G(H,r)$,
so it corresponds to the cylindrical hypersurface $r=r_H$ in region {\tt I\!I'}, called $\Sigma_{s_1}$.
The $r$-component of $\gamma(s_2)$ is $r=r_{+}$ so that $T$-intercept of the corresponding {\small TSS-CMC}
hypersurface $\Sigma_{s_2}$ is $T=0$.
The point $\gamma(s_3)$ achieves the maximum value of the function $F(H,r)$,
so it corresponds to the cylindrical hypersurface $r=R_H$ in region {\tt I\!I}, called $\Sigma_{s_3}$.
The point $\gamma(T)$ corresponds to a {\small TSS-CMC} hypersurface $\Sigma_T$ at the top of the pair of Reissner-Nordstr\"{o}m spacetimes.
Notice that $\Sigma_0$ and $\Sigma_T$ have the same shape but with different $T$-intercepts in the Penrose diagram.

Consider $\cup_{0\leq s\leq T}\Sigma_s$,
where $\Sigma_s$ is the {\small TSS-CMC} hypersurface which corresponds to the point $\gamma(s)$,
then $\cup_{0\leq s\leq T}\Sigma_t$ will foliate the Reissner-Nordstr\"{o}m spacetimes between $\Sigma_0$ and $\Sigma_T$.
The argument is as follows:
(i) Suppose that $(T_0,X_0,\theta,\phi)$ and $(T_0,-X_0,\theta,\phi)$ lie between $\Sigma_0$ and $\Sigma_T$.
By the existence of the Dirichlet problem for the {\small SS-CMC} equation, there exists a {\small TSS-CMC} hypersurface $\Sigma$ passing through them.
Next, $\Sigma$ must have the causality relation that $\Sigma_0\ll \Sigma\ll \Sigma_T$.
Suppose not, say $\Sigma_0\not\ll \Sigma$ for example, then they intersect at some symmetric boundary points or at $X=0$.
The former case will contradict to the uniqueness of the Dirichlet problem for the {\small SS-CMC} equation
and the later case will contradict to the uniqueness of the initial value problem for the {\small SS-CMC} equation.
(ii) Furthermore, this causality argument are also applied to prove that any two hypersurfaces $\Sigma_{s_1}$ and $\Sigma_{s_2}$ in the family are disjoint.

Finally, we copy the family of hypersurfaces $\{\Sigma_s\}_{0\leq s\leq T}$ and paste it to other pairs of the Reissner-Nordstr\"{o}m spacetimes.
In other words, consider $\{\Sigma_s\}_{(k-1)T\leq s\leq kT}$, where $k\in\mathbb{Z}$,
then $\cup_{(k-1)T\leq s\leq kT}\Sigma_t$ will foliate the Reissner-Nordstr\"{o}m spacetimes between $\Sigma_{(k-1)T}$ and $\Sigma_{kT}$.
Therefore, $\cup_{s\in\mathbb{R}}\Sigma_s$ is a {\small TSS-CMC} foliation in the maximally extended Reissner-Nordstr\"{o}m spacetime.
\end{proof}

If we want to construct a {\small TSS-CMC} foliation with varied constant mean curvature in each slice,
the copy and paste method in Theorem~\ref{CMCfoliationFixH} does not work.
However, one-to-one correspondence between a {\small TSS-CMC} hypersurface with mean curvature $H$ and a point on the graph of
$F(H,r)$ or $G(H,r)$ still holds and this is a crucial observation.
Here we will view $F(H,r)$ and $G(H,r)$ as two variables functions,
and we will piecewisely analyze this {\small CMC} foliation problem.

Look at the function $F(H,r)$ first.
When $H$ goes over all real numbers, the range of $F(H,r)$ will cover the strip $r_{-}\leq r\leq r_{+}$ in the $rc$-plane.
In other words, given a point $(r_0,c_0)$ in the strip $r_{-}\leq r\leq r_{+}$,
there is a unique mean curvature $H_0\in\mathbb{R}$ such that $c_0=F(H_0,r_0)$.
Notice that $c_0=F(H_0,r_0)$ will determine the $T$-intercept of a {\small TSS-CMC} hypersurface with mean curvature $H_0$ in region {\tt I\!I}.

In order to construct a {\small TSS-CMC} foliation with varied $H$ in the spacelike extension of the region {\tt I\!I},
we have to find a curve $\gamma(c)$ in the strip $r_{-}\leq r\leq r_{+}$ satisfying the following properties:
\begin{itemize}
\item[(a)] For fixed $H$, $\gamma(c)$ intersects the graph of $F(H,r)$ only once.
This property will imply that any two points on the curves correspond to two {\small TSS-CMC} hypersurfaces with different constant mean curvatures.
\item[(b)] The curve $\gamma(c)$ is a graph of some monotonic function in the $rc$-plane.
The monotonicity property will imply that the constant mean curvature $H(c)$ of corresponding {\small TSS-CMC} hypersurface is monotonic along the
$T$-direction (with respect to $c$).
\item[(c)] In order to get a {\small TSS-CMC} foliation on the whole spacetime, we need to continuously glue the curve $\gamma(c)$
to another curves corresponding to {\small TSS-CMC} hypersurfaces in different regions,
so we require some behavior of the curve when $r$ tends to $r_{-}$ or $r_{+}$.
\end{itemize}
Based on these observations, we derive Proposition~\ref{Proposition1}.
We put the proof of Proposition~\ref{Proposition1} in Appendix, and we label each equation according to the required properties listed above.

\begin{prop} \label{Proposition1}
For the function $F(H,r)=Hr^3+r(-r^2+2Mr-e^2)^{\frac12}$ where $r\in[r_{-},r_{+}]$ and $H\geq 0$,
there exists a function $y(r)$ defined on $[r_{-},r_{+}]$ satisfying
\begin{align}
\left\{
\begin{array}{ll}
\mathrm{(a)} & \displaystyle\frac{\mathrm{d}y}{\mathrm{d}r}\neq\frac{3y}{r}-\frac{(-r^2+3Mr-2e^2)}{(-r^2+2Mr-e^2)^{\frac12}} \\[3mm]
\mathrm{(b)} & \displaystyle\frac{\mathrm{d}y}{\mathrm{d}r}<0 \mbox{ for all } r\in(r_{-},r_{+}) \\[2mm]
\mathrm{(c)} & \displaystyle\lim\limits_{r\to r_{-}}y(r) \mbox{ and } \lim\limits_{r\to r_{+}}y(r)=y(r_+)\geq 0 \mbox{ are finite}.
\end{array}
\right. \label{Conditions}
\end{align}
\end{prop}
Next, we look at the foliation in region {\tt I\!I'}.
When $H$ goes over all real numbers, the range of $G(H,r)$ is also the strip $r_{-}\leq r\leq r_{+}$ in the $rc$-plane.
That is, given a point $(r_1,c_1)$ in the strip $r_{-}\leq r\leq r_{+}$, there is $H_1\in\mathbb{R}$ such that $c_1=G(H_1,r_1)$,
and $G(H_1,r_1)$ will determine the $T$-intercept of a {\small TSS-CMC} hypersurface with mean curvature $H_1$ in region {\tt I\!I'}.

In order to construct a {\small TSS-CMC} foliation with varied $H$ in the spacelike extension of the region {\tt I\!I'},
we have to find a curve $\gamma(c)$ in the strip $r_{-}\leq r\leq r_{+}$ satisfying the properties (a) (change $F(H,r)$ to $G(H,r)$), (b), and (c).
We formulate these properties as Proposition~\ref{Proposition2} and the proof can be found in Appendix.
\begin{prop} \label{Proposition2}
For the function $G(H,r)=Hr^3-r(-r^2+2Mr-e^2)^{\frac12}$ where $r\in[r_{-},r_{+}]$ and $H\geq 0$,
there exists a function $y(r)$ defined on $[r_{-},r_{+}]$ satisfying
\begin{align}
\left\{
\begin{array}{ll}
\mathrm{(a)} & \displaystyle\frac{\mathrm{d}y}{\mathrm{d}r}\neq\frac{3y}{r}+\frac{(-r^2+3Mr-2e^2)}{(-r^2+2Mr-e^2)^{\frac12}} \\[3mm]
\mathrm{(b)} & \displaystyle\frac{\mathrm{d}y}{\mathrm{d}r}>0 \mbox{ for all } r\in(r_{-},r_{+}) \\[2mm]
\mathrm{(c)} & \displaystyle\lim\limits_{r\to r_{-}}y(r)=y(r_{-})>0 \mbox{ and } \lim\limits_{r\to r_{+}}y(r) \mbox{ are finite}.
\end{array}
\right. \label{Conditions2}
\end{align}
\end{prop}

Now we are ready to construct a {\small TSS-CMC} foliation with varied mean curvature in such slice.
\begin{thm} \label{CMCvariedHfoliation}
There is a {\small TSS-CMC} foliation $\{\Sigma_c:(T=T_{c}(X),X,\theta,\phi)\}_{c\in\mathbb{R}}$
in the maximally extended Reissner-Nordstr\"{o}m spacetime such that constant mean curvature
$H(c)$ is increasing along the $T$-direction {\rm(}with respect to $c${\rm)}.
Moreover, this family of {\small TSS-CMC} hypersurfaces is symmetric with respect to the $X$-axis, that is, $T_{-c}(X)=-T_{c}(X)$,
and the mean curvature ranges from $-\infty$ to $\infty$.
\end{thm}

\begin{proof}
It suffices to find a family of {\small TSS-CMC} hypersurfaces $\{\Sigma_c\}_{c\geq 0}$
with nonnegative mean curvature, and $H(c)$ is increasing from $0$ to $\infty$ with respect to $c$,
and $\{\Sigma_c\}_{c\geq 0}$ foliates the maximally extended Reissner-Nordstr\"{o}m spacetime on the region $T\geq 0$.
By the $X$-axis reflection, that is, consider $\{\Sigma_{-c}:(T=T_{-c}(X)\stackrel{\tiny\mbox{def.}}{=}-T_{c}(X),X,\theta,\phi)\}_{c\geq 0}$,
then $\{\Sigma_{c}\}_{c\geq 0}\cup\{\Sigma_{-c}\}_{c\geq 0}$ will foliate the maximally extended Reissner-Nordstr\"{o}m spacetime.

Taking $\Sigma_0:(T=T_0(X)\equiv 0, X,\theta,\phi)$,
then $\Sigma_0$ is a {\small TSS-CMC} hypersurface with mean curvature $H=0$,
and $\Sigma_0$ will correspond to $\gamma(0)=(r_+,c=F(0,r_{+})=0)$ in the $rc$-plane.
Starting from $c=F(0,r_{+})=0$, by Proposition~\ref{Proposition1},
there exists a function $y_1(r)$ satisfying (\ref{Conditions}) and $\lim\limits_{r\to r_{+}}y_1(r)=F(0,r_{+})=0$.
Let $c_1=\lim\limits_{r\to r_{-}}y_1(r)$,
then there exists $H=H(c_1)$ such that $c_1=F(H(c_1),r_{-})=G(H(c_1),r_{-})$.
The graph of $y_1(r)$ is denoted by $\gamma_1(c)=(r,c=y_1(r))$.

Next, from $c_1=G(H(c_1),r_{-})$, by Proposition~\ref{Proposition2},
there exists a function $y_2(r)$ satisfying (\ref{Conditions2}) and $\lim\limits_{r\to r_{-}}y_2(r)=c_1$.
Let $c_2=\lim\limits_{r\to r_{+}}y_2(r)$,
then there exists $H=H(c_2)$ such that $c_2=G(H(c_2),r_{+})=F(H(c_2),r_{+})$.
The graph of $y_2(r)$ is denoted by $\gamma_2(c)=(r,c=y_2(r))$.

We continuous this process.
For $k\in\mathbb{N}$, from the value $c_{2k}$,
by Proposition~\ref{Proposition1},
there exists a function $y_{2k+1}(r)$ satisfying (\ref{Conditions}) and $\lim\limits_{r\to r_{+}}y_{2k+1}(r)=c_{2k}$.
The graph of $y_{2k+1}(r)$ is denoted by $\gamma_{2k+1}(c)=(r,c=y_{2k+1}(r))$.
Let $c_{2k+1}=\lim\limits_{r\to r_{-}}y_{2k+1}(r)$.
From the value $c_{2k+1}$,
by Proposition~\ref{Proposition2},
there exists a function $y_{2k+2}(r)$ satisfying (\ref{Conditions2}) and $\lim\limits_{r\to r_{-}}y_{2k+2}(r)=c_{2k+1}$.
The graph of $y_{2k+2}(r)$ is denoted by $\gamma_{2k+2}(c)=(r,c=y_{2k+2}(r))$, and let $c_{2k+2}=\lim\limits_{r\to r_{+}}y_{2k+2}(r)$.

Let $\gamma(c)=\cup_{i=1}^\infty\gamma_i(c)$,
then the corresponding {\small TSS-CMC} hypersurfaces $\{\Sigma_c\}_{c\geq 0}$ will foliate the Reissner-Nordstr\"{o}m spacetime with $T\geq 0$.

\begin{figure}[h]
\centering
\psfrag{r}{$r$}
\psfrag{c}{$c$}
\psfrag{A}{$\gamma_1(c)$}
\psfrag{B}{$\gamma_2(c)$}
\psfrag{P}{$r=r_{-}$}
\psfrag{Q}{$r=r_{+}$}
\includegraphics[height=62mm, width=30mm]{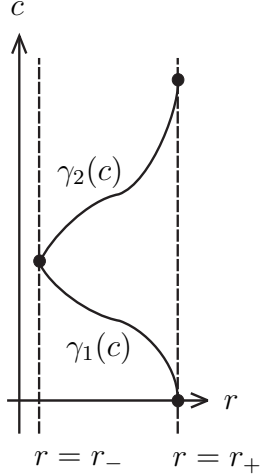}
\caption{We first find a curve $\gamma_1(c)=(r,c=y_1(r))$ satisfying Proposition~\ref{Proposition1} and $y_1(r_+)=0$,
then find another curve $\gamma_2(c)=(r,c=y_2(r))$ satisfying Proposition~\ref{Proposition2} and $y_2(r_-)=y_1(r_{-})$.
We continue the process to get $\cup_{i=1}^\infty\gamma_i(c)$ and the corresponding {\small TSS-CMC}
hypersurfaces will foliate the extended Reissner-Nordstr\"{o}m spacetime on $T\geq 0$.
By $X$-axis reflection, we will get a {\small TSS-CMC} foliation on the whole spacetime.} \label{CMCinRNpaperFoliation}
\end{figure}
\end{proof}

\newpage
\section{Conclusions} \label{Conclusion}
The Reissner-Nordstr\"{o}m spacetime and the Schwarzschild spacetime share the same metric form
$\mathrm{d}s^2=-h(r)\,\mathrm{d}t^2+\frac{1}{h(r)}\,\mathrm{d}r^2+r^2(\mathrm{d}\theta^2+\sin^2\theta\,\mathrm{d}\phi^2)$,
where $h(r)=1-\frac{2M}{r}+\frac{e^2}{r^2}$ and $h(r)=1-\frac{2M}{r}$, respectively.
When the charge smaller than the mass,
it is natural to ask {\small CMC} related questions in the Reissner-Nordstr\"{o}m spacetime metric
when we know fruitful results on the {\small CMC} properties in the Schwarzschild spacetime.
This paper will discuss the initial value problem, Dirichlet problem for the {\small CMC} equation,
and {\small CMC} foliations with fixed mean curvature or varied mean curvature in each slice in the Reissner-Nordstr\"{o}m spacetime.

In Schwarzschild case \cite{LL1}, {\small SS-CMC} equations and their solutions in standard coordinates can be written in terms of $h(r)$,
so solutions of the {\small SS-CMC} equation in the Reissner-Nordstr\"{o}m spacetime are similarly derived by changing $h(r)$ to $1-\frac{2M}{r}+\frac{e^2}{r^2}$.
Two {\small SS-CMC} hypersurfaces can be smoothly glued at every coordinate singularity by checking the blow up order and the next order of the function $f(r)$,
the solution of the {\small SS-CMC} equations, which is similar to the Schwarzschild case.
The only difference in the Reissner-Nordstr\"{o}m spacetime case is that we have to take care of {\small SS-CMC}
hypersurfaces in regions {\tt I\!I\!I} and {\tt I\!I\!I'}, and our analyses also work in these regions.
Thus, the initial value problem is solved in both existence and uniqueness.

Dirichlet problem for the {\small SS-CMC} equation in the Reissner-Nordstr\"{o}m spacetime is also solvable after proving the initial value problem.
Shooting method also works to prove the existence part.
For uniqueness part, geometric analysis method provides a powerful tool.
We can argument it by contradiction: For different {\small TSS-CMC} hypersurfaces satisfying the Dirichlet problem,
consider the maximal timelike geodesic between two hypersurfaces,
then the Laplacian of the Lorentzian distance function of this geodesic gives the relation between two mean curvatures of hypersurfaces.
We can find the inequality is not consistent.
This argument can be applied to the spacetime with timelike convergence condition.

To show the {\small TSS-CMC} foliation properties in the Reissner-Nordstr\"{o}m spacetime with fixed mean curvature,
they are highly related to the existence and uniqueness of the Dirichlet problem for {\small SS-CMC} equations with symmetric boundary data.
This is because the existence part is equivalent that hypersurfaces cover the spacetime.
The uniqueness part is equivalent that any two hypersurfaces are disjoint.
This equivalent statement is also established in the Schwarzschild spacetime case \cite[Theorem 8]{KWL}.

Since the extended Reissner-Nordstr\"{o}m spacetime consists of infinity many pairs of Reissner-Nordstr\"{o}m spacetimes,
for {\small TSS-CMC} foliation with fixed mean curvature, we can choose the {\small TSS-CMC} hypersurfaces with periodicity.
However,
we have to be careful to choose special {\small TSS-CMC} hypersurfaces with varied mean curvature to foliate the extended Reissner-Nordstr\"{o}m spacetime.
This {\small TSS-CMC} foliation problem is solved because we can find a one-to-one correspondence for every {\small TSS-CMC} hypersurface to the point in the
$rc$-plane with $r_{-}\leq r\leq r_{+}$,
which detects the position of the {\small TSS-CMC} hypersurface on the $T$-axis.
So we can change every required {\small TSS-CMC} foliation properties to equivalent conditions (a)--(c) in Proposition~\ref{Proposition1}
and Proposition~\ref{Proposition2} and show the existence of a function satisfying (a)--(c).
Thus, {\small TSS-CMC} foliation with varied $H$ is established.

\section*{Appendix}
We will give proofs of Proposition~\ref{Proposition1} and Proposition~\ref{Proposition2} here.
\setcounter{thm}{3}
\begin{prop}
For the function $F(H,r)=Hr^3+r(-r^2+2Mr-e^2)^{\frac12}$ where $r\in[r_{-},r_{+}]$ and $H\geq 0$,
there exists a function $y(r)$ defined on $[r_{-},r_{+}]$ satisfying
\begin{align*}
\left\{
\begin{array}{ll}
\mathrm{(a)} & \displaystyle\frac{\mathrm{d}y}{\mathrm{d}r}\neq\frac{3y}{r}-\frac{(-r^2+3Mr-2e^2)}{(-r^2+2Mr-e^2)^{\frac12}} \\[3mm]
\mathrm{(b)} & \displaystyle\frac{\mathrm{d}y}{\mathrm{d}r}<0 \mbox{ for all } r\in(r_{-},r_{+}) \\[2mm]
\mathrm{(c)} & \displaystyle \lim\limits_{r\to r_{-}}y(r) \mbox{ and } \lim\limits_{r\to r_{+}}y(r)=y(r_{+})\geq 0 \mbox{ are finite}.
\end{array}
\right.
\end{align*}
\end{prop}

\begin{proof}
First of all, we compute
\begin{align*}
\frac{\partial F}{\partial r}(H,r)
=3Hr^2+\frac{(-2r^2+3Mr-e^2)}{(-r^2+2Mr-e^2)^{\frac12}}
=\frac{3y}{r}-\frac{(-r^2+3Mr-2e^2)}{(-r^2+2Mr-e^2)^{\frac12}}.
\end{align*}
Here we replace $H$ with $y$ and $r$ by the relation $y=Hr^3+r(-r^2+2Mr-e^2)^{\frac12}$ in the last equality,
and this will be the right hand side of (a).
In order to find a function $y(r)$ satisfying (a), it suffices to find a function $p(r)>0$ such that
\begin{align*}
\displaystyle\frac{\mathrm{d}y}{\mathrm{d}r}-\frac{3y}{r}=-\frac{(-r^2+3Mr-2e^2)}{(-r^2+2Mr-e^2)^{\frac12}}-p(r).
\end{align*}
When multiplying the integrating factor $\mathrm{e}^{\int-\frac3r\,\mathrm{d}r}=r^{-3}$ on both sides of the above differential equation, it becomes
\begin{align*}
\frac{\mathrm{d}}{\mathrm{d}r}\left(r^{-3}y(r)\right)=-\frac{(-r^2+3Mr-2e^2)}{r^3(-r^2+2Mr-e^2)^{\frac12}}-\frac{p(r)}{r^3}.
\end{align*}
The function $y(r)$ is solved by choosing the initial value of the integration $y(r_{+})\geq 0$:
\begin{align}
y(r)=\frac{y(r_{+})}{r_{+}^3}r^3+r(-r^2+2Mr-e^2)^{\frac12}+r^3\int_{r}^{r_+}\frac{p(x)}{x^3}\,\mathrm{d}x. \label{functiony}
\end{align}
Next, we hope that the function $y(r)$ has property (b), so we compute
\begin{align*}
y'(r)=\frac{3y(r_{+})}{r_{+}^3}r^2+\frac{(-2r^2+3Mr-e^2)}{(-r^2+2Mr-e^2)^{\frac12}}+3r^2\int_{r}^{r_+}\frac{p(x)}{x^3}\,\mathrm{d}x-p(r).
\end{align*}
Notice that the root of $q(r)\stackrel{\tiny\mbox{def.}}{=}\frac{(-2r^2+3Mr-e^2)}{(-r^2+2Mr-e^2)^{\frac12}}=0$ is $r_{*}=\frac{3M+\sqrt{9M^2-8e^2}}{4}$
and it implies that $q(r)>0$ on $(r_{-},r_{*})$ and $q(r)<0$ on $(r_{*},r_{+})$.
Consider
\begin{align*}
p(r)=q(r)\cdot\chi_{(r_{-},r_{*})}(r)+Cr^{\alpha}=\frac{(-2r^2+3Mr-e^2)}{(-r^2+2Mr-e^2)^{\frac12}}\cdot \chi_{(r_{-},r_{*})}(r)+Cr^{\alpha},
\end{align*}
where $\chi_{(r_{-},r_{*})}(r)$ is the characteristic function (indicator function) of the set $(r_{-},r_{*})$, and $\alpha<-1$ is a fixed constant,
and $C$ is a constant to be determined later. Since
\begin{align*}
3r^2\int_{r}^{r_+}\frac{p(x)}{x^3}\,\mathrm{d}x
&=3r^2\int_{r}^{r_+}\frac{-2x^2+3Mx-e^2}{x^3(-x^2+2Mx-e^2)^{\frac12}}\cdot \chi_{(r_{-},r_{*})}(x)\,\mathrm{d}x+3Cr^2\int_{r}^{r_+}x^{\alpha-3}\,\mathrm{d}x \\
&\leq 3C_1r^2-\frac{3Cr^2}{2-\alpha}\left(\frac{1}{r_{+}^{2-\alpha}}-\frac{1}{r^{2-\alpha}}\right),
\end{align*}
where $C_1=\int_{r_{-}}^{r_+}\frac{-2x^2+3Mx-e^2}{x^3(-x^2+2Mx-e^2)^{\frac12}}\cdot\chi_{(r_{-},r_{*})}(x)\,\mathrm{d}x$
is a finite value, we have
\begin{align*}
y'(r)
&\leq\frac{3y(r_{+})}{r_{+}^3}r^2+3C_1r^2-\frac{3Cr^2}{2-\alpha}\left(\frac{1}{r_{+}^{2-\alpha}}-\frac{1}{r^{2-\alpha}}\right)-Cr^{\alpha}
+(1-\chi_{(r_{-},r_{*})}(r))q(r) \\
&=\left(\frac{3y(r_{+})}{r_{+}^3}+3C_1-\frac{3C}{2-\alpha}\frac{1}{r_{+}^{2-\alpha}}\right)r^2+\left(\frac{3}{2-\alpha}-1\right)Cr^{\alpha}
+(1-\chi_{(r_{-},r_{*})}(r))q(r).
\end{align*}
Since $\alpha<-1$, we can choose $C$ large enough so that $y'(r)<0$ for all $r\in(r_{-},r_{+})$.

Finally, since $p(r)$ is of order $O((r-r_{-})^{-\frac12})$, the improper integral in (\ref{functiony}) is finite,
so $\lim\limits_{r\to r_{-}}y(r)$ is finite.
The other limit $\lim\limits_{r\to r_{+}}y(r)=y(r_{+})$ is finite because of the initial value setting in (\ref{functiony}).
\end{proof}

\begin{prop}
For the function $G(H,r)=Hr^3-r(-r^2+2Mr-e^2)^{\frac12}$ where $r\in[r_{-},r_{+}]$ and $H\geq 0$,
there exists a function $y(r)$ defined on $[r_{-},r_{+}]$ satisfying
\begin{align*}
\left\{
\begin{array}{ll}
\mathrm{(a)} & \displaystyle\frac{\mathrm{d}y}{\mathrm{d}r}\neq\frac{3y}{r}+\frac{(-r^2+3Mr-2e^2)}{(-r^2+2Mr-e^2)^{\frac12}} \\[3mm]
\mathrm{(b)} & \displaystyle\frac{\mathrm{d}y}{\mathrm{d}r}>0 \mbox{ for all } r\in(r_{-},r_{+}) \\[2mm]
\mathrm{(c)} & \displaystyle \lim\limits_{r\to r_{-}}y(r)=y(r_{-})>0 \mbox{ and } \lim\limits_{r\to r_{+}}y(r) \mbox{ are finite}.
\end{array}
\right.
\end{align*}
\end{prop}

\begin{proof}
First of all, we compute
\begin{align*}
\frac{\partial G}{\partial r}(H,r)
=3Hr^2-\frac{(-2r^2+3Mr-e^2)}{(-r^2+2Mr-e^2)^{\frac12}}
=\frac{3y}{r}+\frac{(-r^2+3Mr-2e^2)}{(-r^2+2Mr-e^2)^{\frac12}}.
\end{align*}
Here we replace $H$ with $y$ and $r$ by the relation $y=Hr^3-r(-r^2+2Mr-e^2)^{\frac12}$ in the last equality,
and this will be the right hand side of (a).
In order to find a function $y(r)$ satisfying (a), it suffices to find a function $p(r)>0$ such that
\begin{align*}
\displaystyle\frac{\mathrm{d}y}{\mathrm{d}r}-\frac{3y}{r}=\frac{(-r^2+3Mr-2e^2)}{(-r^2+2Mr-e^2)^{\frac12}}+p(r).
\end{align*}
When multiplying the integrating factor $\mathrm{e}^{\int-\frac3r\,\mathrm{d}r}=r^{-3}$ on both sides of the differential equation, it becomes
\begin{align*}
\frac{\mathrm{d}}{\mathrm{d}r}\left(r^{-3}y(r)\right)=\frac{(-r^2+3Mr-2e^2)}{r^3(-r^2+2Mr-e^2)^{\frac12}}+\frac{p(r)}{r^3}.
\end{align*}
The function $y(r)$ is solved by choosing the initial value of the integration $y(r_{-})>0$:
\begin{align}
y(r)=\frac{y(r_{-})}{r_{-}^3}r^3-r(-r^2+2Mr-e^2)^{\frac12}+r^3\int_{r_{-}}^{r}\frac{p(x)}{x^3}\,\mathrm{d}x. \label{functiony2}
\end{align}
Next, we hope that the function $y(r)$ has property (b), so we compute
\begin{align*}
y'(r)=\frac{3y(r_{-})}{r_{-}^3}r^2-\frac{(-2r^2+3Mr-e^2)}{(-r^2+2Mr-e^2)^{\frac12}}+3r^2\int_{r_{-}}^{r}\frac{p(x)}{x^3}\,\mathrm{d}x+p(r).
\end{align*}
Notice that the root of $q(r)\stackrel{\tiny\mbox{def.}}{=}\frac{(-2r^2+3Mr-e^2)}{(-r^2+2Mr-e^2)^{\frac12}}=0$ is $r_{*}=\frac{3M+\sqrt{9M^2-8e^2}}{4}$
and it implies that $q(r)>0$ on $(r_{-},r_{*})$ and $q(r)<0$ on $(r_{*},r_{+})$.
Consider
\begin{align*}
p(r)=q(r)\cdot\chi_{(r_{-},r_{*})}(r)=\frac{(-2r^2+3Mr-e^2)}{(-r^2+2Mr-e^2)^{\frac12}}\cdot \chi_{(r_{-},r_{*})}(r)\geq 0,
\end{align*}
where $\chi_{(r_{-},r_{*})}(r)$ is the characteristic function (indicator function) of the set $(r_{-},r_{*})$.
Then we have
\begin{align*}
y'(r)=\frac{3y(r_{-})}{r_{-}^3}r^2+3r^2\int_{r_{-}}^{r}\frac{p(x)}{x^3}\,\mathrm{d}x+(\chi_{(r_{-},r_{*})}(r)-1)q(r)>0.
\end{align*}
Finally, since $p(r)$ is of order $O(|r-r_{+}|^{-\frac12})$, the improper integral in (\ref{functiony2}) is finite,
so $\lim\limits_{r\to r_{+}}y(r)$ is finite.
The other limit $\lim\limits_{r\to r_{-}}y(r)=y(r_{-})$ is finite because of the initial value setting in (\ref{functiony2}).
\end{proof}

\section*{Acknowledgments}
The author would like to thank Yng-Ing Lee and Pengzi Miao for their
interests and suggestions. The author is supported by MOST, Grant
No.106-2115-M-018-002-MY2.

\fontsize{11}{14pt plus.5pt minus.4pt}\selectfont

\vspace*{5mm}
\noindent{Kuo-Wei Lee}\\
\noindent{Department of Mathematics, National Changhua University of Education, Changhua, Taiwan}\\
\noindent{E-mail: \url{kwlee@cc.ncue.edu.tw; d93221007@gmail.com}}
\end{document}